\documentclass{amsart}

\usepackage{amsmath,amssymb,yhmath,todonotes,tikz-cd}
\usepackage[all]{xy}
\usepackage[colorlinks=true,citecolor=blue,linkcolor=blue,urlcolor=blue]{hyperref}

\newtheorem{thm}{Theorem}[section]
\newtheorem{lem}[thm]{Lemma}
\newtheorem{cor}[thm]{Corollary}
\newtheorem{prop}[thm]{Proposition}
\newtheorem{conj}[thm]{Conjecture}
\newtheorem{remark}[thm]{Remark}

\newtheorem{introthm}{Theorem}

\numberwithin{equation}{section}

\newcommand{\Dcap}{\wideparen{\mathcal D}}
\newcommand{\C}{\mathcal C}
\newcommand{\D}{\mathcal D}
\newcommand{\E}{\mathcal E}
\newcommand{\M}{\mathcal M}
\newcommand{\N}{\mathcal N}
\renewcommand{\O}{\mathcal O}
\newcommand{\cD}{\mathcal D}
\newcommand{\cE}{\mathcal E}
\newcommand{\bP}{\mathbb P}
\newcommand{\bQ}{\mathbb Q}
\newcommand{\Sp}{\operatorname{Sp}}
\newcommand{\dR}{\mathrm{dR}}
\newcommand{\an}{\mathrm{an}}
\newcommand{\ZZ}{\mathbb{Z}}
\newcommand{\bL}{\mathbb{L}}
\DeclareMathOperator{\proet}{\text{pro\'{e}t}} 
\DeclareMathOperator{\et}{\text{\'{e}t}} 
\DeclareMathOperator{\ket}{\text{k\'{e}t}} 
\DeclareMathOperator{\OB}{\mathcal{O}\kern -1,00pt\mathbb{B}} 

\begin{document}

\title[Gauss--Manin connections extend to holonomic $\Dcap$-modules]{Gauss--Manin connections extend to holonomic coadmissible $\Dcap$-modules}

\author{Andreas Bode}
\address{Bergische Universit{\"a}t Wuppertal, Gau{\ss}stra{\ss}e 20, D-42119 Wuppertal, Germany}
\email{abode@uni-wuppertal.de}

\author{Finn Wiersig}
\address{National University of Singapore, 10 Lower Kent Ridge Rd, Singapore 119260, Singapore}
\email{fwiersig@nus.edu.sg}
\date{\today}

\begin{abstract}
We prove that the pushforwards of Gauss--Manin connections on smooth
rigid analytic varieties along Zariski-open immersions are coadmissible and
holonomic over Ardakov--Wadsley's sheaf $\Dcap$
of infinite order differential operators. This may be viewed as a rigid analytic variant of the classical theorem of
Griffiths, Deligne, and Katz that Gauss-Manin systems have regular singularities
after compactification. The proof combines $p$-adic de Rham comparison and
Diao--Lan--Liu--Zhu's logarithmic Riemann--Hilbert correspondence with extendability results for log-connections with nilpotent residues. In the algebraic snc case, we also establish weak holonomicity via a
$p$-adic Bernstein-Sato criterion of Bitoun and the first author.
\end{abstract}



\maketitle
\tableofcontents

\section{Introduction}

Gauss--Manin connections are a basic source of differential equations in
algebraic geometry. In the framework of algebraic de Rham cohomology
initiated by Grothendieck \cite{GrothendieckDeRham}, a smooth proper
morphism \(f\colon Y\to U\) gives rise to relative de Rham cohomology
sheaves $R^q f_*\Omega^\bullet_{Y/U}$,
and these sheaves carry a canonical integrable connection. An algebraic construction of this connection was given by Katz--Oda via the filtered absolute de Rham
complex and the associated spectral sequence \cite{KatzOda}; see also Katz's
work on period matrices \cite{KatzPeriodMatrices}.

Over the complex numbers, the same construction has a classical analytic
interpretation: its horizontal sections are period integrals, and in
one-parameter families these periods satisfy the Picard--Fuchs differential
equations. This
viewpoint goes back to the classical theory of periods and was developed in
modern form by Manin, Griffiths, Katz, and others
\cite{ManinDifferentiation,GriffithsPeriodsI,GriffithsPeriodsII,
KatzPeriodMatrices}.

A fundamental feature of the classical theory is the regularity of
Gauss--Manin connections at infinity. After compactifying the base, the
singularities of a Gauss--Manin connection along the boundary are regular;
equivalently, when the boundary is a normal crossings divisor, the connection
admits a logarithmic extension. This regularity theorem for Picard--Fuchs
equations was proved analytically by Griffiths and Deligne, and
algebraically by Katz; see in particular
\cite{GriffithsPeriodsI,GriffithsPeriodsII,KatzRegularity,
DeligneRegularSingularities}. 

In the language of algebraic \(\mathcal D\)-modules, this means that
Gauss--Manin systems remain regular holonomic after compactifying the base:
if \(j\colon U\hookrightarrow X\) is a smooth compactification, then
the cohomology sheaves of the direct image \(j_+\mathcal E^q\cong Rj_{*}\mathcal E^{q}\)
are regular holonomic \(\mathcal D_{X}\)-modules. Thus the classical
theory singles out Gauss--Manin connections as differential equations of
geometric origin whose behaviour near the boundary is much better than that
of an arbitrary integrable connection.

The purpose of this paper is to prove a rigid analytic analogue of this
finiteness phenomenon. We fix a complete discretely valued field \(K\) of
mixed characteristic \((0,p)\) with perfect residue field. Let
$j\colon U\hookrightarrow X$
be a Zariski-open immersion of smooth rigid analytic \(K\)-spaces, and let
$f\colon Y\to U$
be proper and smooth. For each \(q\geq 0\), we write
$\mathcal E^q := R^q f_*\Omega^\bullet_{Y/U}$
for the \(q\)-th Gauss--Manin connection on \(U\). We are interested in the
behaviour of \(\mathcal E^q\) along the boundary \(Z=X\setminus U\), or,
more precisely, in the finiteness properties of the higher direct image sheaves
$R^{i} j_*\mathcal E^q$.

A naive translation of the algebraic statement using only finite-order
differential operators does not give a satisfactory finiteness theory in the
rigid analytic setting. Already for the trivial connection \(\mathcal O_U\), the
sheaf \(j_*\mathcal O_U\) is not coherent as a module over the
usual sheaf \(\mathcal D_X\) of finite-order differential operators. Thus the
rigid analytic situation differs sharply from the classical algebraic one, and
one is led to work with a larger sheaf of differential operators.

The appropriate sheaf for our purposes is Ardakov--Wadsley's sheaf
\(\Dcap_X\) of rapidly converging, infinite-order differential operators
\cite{AW19,MR3846550}. This construction is closely related in spirit to
Berthelot's theory of arithmetic \(\mathcal D\)-modules
\cite{BerthelotDmodulesArithmeticI} and to subsequent work of Huyghe and
collaborators~\cite{HuygheDdaggerAffiniteProjectif},
\cite{NootHuygheBB,HuygheSchmidtStrauchArithmeticStructures}.
The sheaf \(\Dcap_X\) contains the usual sheaf \(\mathcal D_X\), but its
module theory is governed by a different finiteness condition. Indeed, the
rings of sections of \(\Dcap_X\) over suitable affinoid subdomains are
Fr\'echet--Stein algebras in the sense of Schneider--Teitelbaum
\cite{ST03}, as shown by Ardakov--Wadsley and, in greater generality, by the
first author \cite{AW19,Compl}. The corresponding finite objects are
the so-called coadmissible modules, which indeed play an analogous role to coherent $\D$-modules in the classical theory.

Within this framework, we establish the following result on the higher pushforwards of Gauss--Manin connections:

\begin{introthm}\label{thm:main}
Let $K$ be a complete discretely valued field of
mixed characteristic \((0,p)\) with perfect residue field,
let $j\colon U\hookrightarrow X$
be a Zariski-open immersion of smooth rigid analytic \(K\)-spaces, and let
$f\colon Y\to U$
be proper and smooth.
Then, for every \(i,q\in \mathbb Z_{\geq 0}\) the higher direct images
\[R^i j_*\mathcal E^q\]
are coadmissible \(\Dcap_X\)-modules.
\end{introthm}

Theorem~\ref{thm:main} is not a
formal consequences of the general theory of
\(\Dcap\)-modules. There exist vector bundles with integrable connection
\(\mathcal E\) on \(U\) for which the sheaves \(R^i j_*\mathcal E\) are not
coadmissible, see \cite{Bitoun}. The content of the theorem is that connections arising from smooth proper geometry have substantially better boundary behaviour than arbitrary integrable connections.

Our second main result concerns a stronger finiteness property than
coadmissibility, namely holonomicity.
Within the category of coadmissible \(\Dcap_X\)-modules, various
analogues of holonomicity have been suggested. Weak holonomicity~\cite{Dcapthree} is a homological condition, modelled on the expected vanishing of Ext-groups outside the middle degree.
There is also a functorial notion of holonomicity~\cite{Hol}, designed to behave well under the operations appearing in the rigid analytic \(\Dcap\)-module formalism. At present it is not known whether holonomicity implies weak holonomicity in this setting, although this is expected.

There is currently no generally accepted notion of regular singularity, or of
regular holonomicity, for \(\Dcap\)-modules.
The theorems below may be viewed
as a step toward such a theory: it shows that Gauss--Manin connections have,
after extension across the boundary, the
strongest possible finiteness properties.

\begin{introthm}\label{introthm:snc}
With the notation as in Theorem~\ref{thm:main}, assume additionally that $U$ is the complement of an algebraic strict normal crossing (snc) divisor in the sense of \cite[Definition 9.2]{Dcapthree}. Then, for every $q\in \mathbb{Z}_{\geq 0}$, the direct images
\begin{equation*}
j_*\mathcal{E}^q
\end{equation*}
are weakly holonomic and holonomic $\Dcap_X$-modules (and $R^ij_*\mathcal{E}^q=0$ for $i>0$).
\end{introthm}

One may reduce the general setting to the snc case via a desingularization argument, but this is most naturally formulated in the derived setting.
We therefore use the formalism of \(\mathcal C\)-complexes as in \cite{6Op}.
For example, a complex of $\Dcap_X$-modules which is concentrated in degree
zero is a $\mathcal{C}$-complex if and only if it is coadmissible.

\begin{introthm}\label{introthm:mainhol}
With the notation as in Theorem~\ref{thm:main},
the complexes
\[ j_+\cE^q, \ j_+ f_+\mathcal O_Y \in \mathrm D\left(\Dcap_X\right)\]
are holonomic \(\mathcal C\)-complexes.
\end{introthm}

Just as before, neither Theorem~\ref{introthm:snc} nor Theorem~\ref{introthm:mainhol} follow formally from the theory of holonomic
$\mathcal{C}$-complexes.
Indeed, there exist vector bundles with integrable connection
\(\mathcal E\) on an snc complement \(U\) for which the sheaf \(j_*\mathcal E\) is
coadmissible, but not holonomic, as in \cite[Lemma 6.2]{Hol}.

While it is natural to expect that the individual higher pushforwards $R^i j_*\mathcal E^q$ are also holonomic and weakly holonomic beyond the snc case, such a statement currently rests on resolving some fundamental open questions on (weak) holonomicity. For instance, it is known that pushforward along a projective morphism preserves holonomic $\mathcal{C}$-complexes by \cite[Theorem 5.10.(iii)]{Hol}, which is how we can go from Theorem~\ref{introthm:snc} to Theorem~\ref{introthm:mainhol}, but stability of weak holonomicity has so far not been established. For holonomicity, it is not clear that each cohomology group of a holonomic complex is again holonomic. 

As far as Gauss--Manin connections are concerned, these questions turn out to be related: similar to the discussion in \cite[Lemma 6.5]{Hol} by the first author, one should be able to show that a bounded holonomic complex with weakly holonomic cohomology groups has holonomic cohomology groups, so that settling the stability of weak holonomicity under projective direct images would suffice to also establish the holonomicity of the individual cohomology groups $R^ij_*\mathcal{E}^q$. We hope to return to this question elsewhere.

\subsection{Proofs of the Main Theorems}

First, we explain the proof of Theorem~\ref{thm:main}. Arguing as in \cite{Dcapthree}, we apply Temkin's desingularisation theorem \cite{TemkinDesingularization} and work locally to assume that \(Z=X\setminus U\) is a strict normal crossing divisor. In this case $j_*$ is exact, thus it suffices to treat the case \(i=0\), i.e. to prove that \(j_*\mathcal E^q\) is a coadmissible \(\Dcap_X\)-module.

By Scholze's relative de Rham--\'etale comparison theorem
\cite[Theorem 8.8(ii)]{Sch13pAdicHodge}, the Gauss--Manin connection \(\mathcal E^q\) is the
de Rham realization of a \(\mathbb Z_p\)-local system on \(U\). We may
therefore apply the logarithmic Riemann--Hilbert correspondence of
Diao--Lan--Liu--Zhu \cite{MR4536903}. It gives a logarithmic connection
\(\mathcal M^q_{\log}\) on \(X\), extending \(\mathcal E^q\), whose residues
along \(Z\) are rational. It is in reference to these results that we need to impose the condition of $K$ being discretely valued with perfect residue field.

Let $\mathcal M^q := \mathcal M^q_{\log}(*Z)$
be the associated meromorphic connection. The rationality of the residues
implies, by the computation of \(b\)-functions in
Section~\ref{sec:log-res-b}, that the relevant \(b\)-functions of
\(\mathcal M^q\) have only rational roots. We can now apply the theorem of
Bitoun and the first author \cite{Bitoun}: the extension
$j_*\mathcal E^q = j_*(\mathcal M^q|_U)$
is a coadmissible, weakly holonomic \(\Dcap_X\)-module. This proves Theorem~\ref{thm:main}, as well as the weakly holonomic part of Theorem~\ref{introthm:snc}.

Now we discuss the proof of Theorem~\ref{introthm:mainhol} and the holonomic part of Theorem~\ref{introthm:snc}. Since holonomic $\C$-complex are stable under projective pushforward, one can again reduce to the snc case. One then passes to a finite Kummer cover which clears the
denominators of the residue eigenvalues. The canonical logarithmic extension
of the pulled-back connection then has nilpotent residues, for which we develop a suitable extendability theory in Section~\ref{sec:nilpotent-residues}. The finite
\'etale charts of Corollary~\ref{cor:finite-etale-boundary-charts},
together with Corollary~\ref{cor:nilpotent-log-holonomic}, give holonomicity on
the cover. The trace splitting and Lemma~\ref{directsummand} descend this
property, yielding the holonomicity statement in Theorem~\ref{introthm:snc}, as well as a proof of Theorem~\ref{introthm:mainhol}, and an alternative proof of Theorem~\ref{thm:main}. The details are given in Section~\ref{sec:main-proofs}.

We remark in particular that the proof strategies for weak holonomicity and holonomicity appear essentially independent from each other: weak holonomicity is a consequence of the theory of $b$-functions from section~\ref{sec:log-res-b}, while holonomicity is a consequence of the discussion in section~\ref{sec:nilpotent-residues} -- no results from section~\ref{sec:log-res-b} are needed for the proof of Theorem~\ref{introthm:mainhol}.

\subsection{Outlook: Gauss--Manin connections on the
\texorpdfstring{$p$}{p}-adic upper half-plane}

Let $F$ be a finite extension of $\bQ_{p}$ contained in $K$, $\Omega_{F}$
the rigid $F$-analytic $p$-adic upper half-plane, and let $G^{0}$ denote the group
of $2\times2$ matrices with entries in $F$ whose determinant has absolute value $1$.
Then $G^{0}$ acts on $\Omega:=\Omega_{F}\times_{F}K$ via M\"obius transformations.
Let $f\colon\Sigma\to\Omega$ be
one of the coverings in the Drinfeld tower~\cite{MR422290}, or,
more generally,
a proper smooth $G^{0}$-equivariant map.
For $q\geq0$, the relative de Rham cohomology
$\cE^{q}:=R^q f_*\Omega^\bullet_{\Sigma/\Omega}$
is a $G^{0}$-equivariant Gauss--Manin connection on $\Omega$.
Let $j\colon\Omega\hookrightarrow\bP:=\bP_{K}^{1,\an}$ be the canonical open immersion,
which is $G^{0}$-equivariant.

\begin{conj}[Ardakov--Wadsley]\label{conj:AW-equiv}
  The pushforwards $j_{*}\cE^{q}$ are coadmissible as $G^{0}$-equivariant
  $\cD$-modules on $\bP$ (and $R^ij_*\cE^q=0$ for $i>0$, as $\Omega$ is Stein).
\end{conj}

Note that if $f$ is a Drinfeld covering, we have $\cE^q=0$ for $q>0$, since $f$ is \'etale. 

Conjecture~\ref{conj:AW-equiv} remains open in this generality. It is
motivated by the expectation that geometric equivariant connections on
$\Omega$ should give rise, through Ardakov's equivariant $\cD$-module
formalism~\cite{EquivariantD2021}, to admissible locally analytic representations of $G^{0}$. Closely related
finiteness phenomena have been proved by Ardakov--Wadsley for
$G^0$-equivariant line bundles with connection on $\Omega$; see
\cite{ArdakovWadsleyEquivariantLineBundles,ArdakovWadsleyGlobalSections}.
We also mention related work of Dospinescu--Le Bras, Pan, and Qiu--Su on
locally analytic vectors and the geometry of Drinfeld towers
\cite{DospinescuLeBrasDrinfeld,PanLocallyAnalyticVectors,PanLocallyAnalyticVectorsII,QiuSuLocallyAnalyticVectors}.
The methods in these works rely either on
Colmez's $p$-adic local Langlands correspondence when
$F=\bQ_p$, or on completed cohomology and global automorphic input.

The aim of the present paper is to isolate the local $\cD$-module-theoretic
part of this picture. If one forgets the $G^0$-action and replaces
\[
\Omega \rightsquigarrow U,\qquad
\Sigma \rightsquigarrow Y,\qquad
\bP \rightsquigarrow X,
\]
then Conjecture~\ref{conj:AW-equiv} becomes precisely the assertion that
the higher pushforwards of Gauss--Manin connections along a Zariski-open
immersion are coadmissible $\Dcap$-modules. This is the content of
Theorem~\ref{thm:main}. Thus the present article proves the non-equivariant
local finiteness statement underlying the expected equivariant theorem on the
$p$-adic upper half-plane.
\subsection*{Acknowledgements}
We thank Konstantin Ardakov and Hansheng Diao for valuable conversations.

The authors used OpenAI's ChatGPT (GPT-5.5 Pro and GPT-5.6 Sol, June--August 2026) to edit and reorganize the manuscript, draft and expand selected expository and proof passages, including the coadmissibility argument and the comparison between residue eigenvalues and roots of $b$-functions in section~\ref{sec:log-res-b}, check arguments for possible gaps, and suggest references and bibliographic metadata. All output was independently reviewed, revised, and verified by the authors, who take full responsibility for the article.

\section{Finiteness preliminaries and divisors}
\label{sec:preliminaries}

\subsection{Coadmissible $\Dcap$-modules}

Given a smooth rigid analytic $K$-variety $X$, we work with the derived category of complete bornological $\Dcap$-modules $\mathrm D(\Dcap_X)$ and its
subcategory $\mathrm D_{\C}(\Dcap_X)$ of $\C$-complexes as in~\cite{6Op}.
A module is coadmissible precisely when it is a $\C$-complex concentrated in
degree zero (\cite[Proposition 8.3]{6Op}). More generally, a bounded complex is a $\C$-complex if and only if each cohomology group is coadmissible \cite[Corollary 8.7]{6Op}. We use the notion of holonomicity introduced in~\cite{Hol}; the holonomic complexes form a triangulated subcategory $\mathrm D_{\mathrm{hol}}(\Dcap_X)\subseteq \mathrm{D}_{\C}(\Dcap_X)$, stable under inverse images and under
projective direct images. We shall need only the following consequences.

\begin{lem}\label{directsummand}
Let $X$ be smooth and let $\M$ be a direct summand of a $\Dcap_X$-module
$\M'$. If $\M'$ is coadmissible, respectively holonomic, then so is $\M$.
\end{lem}

\begin{proof}
It suffices to consider the case where $X$ is a smooth affinoid admitting a local coordinate system, so that we can write $\wideparen{\mathcal{D}}_X\cong \varprojlim \D_n$ as in \cite[section 8]{6Op}.
		
Let $\M^\bullet, {\M''}^\bullet\in \mathrm{D}(\Dcap_X)$ and set ${\M'}^\bullet=\M^\bullet\oplus {\M''}^\bullet$. Suppose that $\M'^\bullet$ is a $\C$-complex. Then
\begin{equation*}
	\D_n\widetilde{\otimes}^{\mathbb{L}}_{\Dcap_X}\M'^\bullet\cong (\D_n\widetilde{\otimes}^{\mathbb{L}}_{\Dcap_X}\M^\bullet)\oplus (\D_n\widetilde{\otimes}^{\mathbb{L}}_{\Dcap_X}\M''^\bullet)\in \mathrm{D}^b_{\mathrm{coh}}(\D_n)
\end{equation*}
by assumption, and hence $\D_n\widetilde{\otimes}^{\mathbb{L}}_{\Dcap_X}\M^\bullet\in \mathrm{D}^b_{\mathrm{coh}}(\D_n)$ for each $n$. Now write $\M_n^\bullet=\D_n\widetilde{\otimes}^{\mathbb{L}}\M^\bullet$, similarly for the other complexes, and consider for each $i$ the commutative diagram
\begin{equation*}
	\begin{xy}
		\xymatrix{0\ar[r]&\mathrm{H}^i(\M^\bullet)\ar[r]\ar[d]&\mathrm{H}^i(\M'^\bullet)\ar[r]\ar[d]&\mathrm{H}^i(\M''^\bullet)\ar[r]\ar[d]&0\\
		0\ar[r]& \varprojlim \mathrm{H}^i(\M_n^\bullet)\ar[r]&\varprojlim \mathrm{H}^i({\M'_n}^\bullet)\ar[r]&\varprojlim \mathrm{H}^i({\M''_n}^\bullet)\ar[r]&0}
	\end{xy}
\end{equation*}
where both rows are exact and split by functoriality. Since the middle arrow is an isomorphism by assumption, so are both outer arrows, and $\M^\bullet$ is a $\C$-complex.
		
In particular, if $\M^\bullet$ and $\M''^\bullet$ are concentrated in degree zero, this proves that direct summands of coadmissible modules are coadmissible.
		
Moreover, since \cite[Definition~5.1]{Hol} characterizes holonomicity as precisely those $\C$-complexes which remain $\C$-complexes under various functors (which all preserve finite direct sums), the above readily implies that direct summands of holonomic modules are holonomic.
\end{proof}

\begin{lem}\label{fettest}
Let \(g\colon Y\to X\) be a finite \'etale morphism between smooth
rigid analytic spaces. A \(\Dcap_Y\)-module \(\M\) is coadmissible,
respectively holonomic, if and only if \(g_+\M\) has the corresponding
property.
\end{lem}

\begin{proof}
For the forward direction, use \cite[Theorem~1.3.(iv)]{6Op} for
coadmissibility and \cite[Theorem~5.10.(iii)]{Hol} for holonomicity.

Conversely, assume $g_+\M$ is coadmissible, i.e. a $\C$-complex concentrated in degree zero. In particular, $g_+\M=g_*\M$, as $g$ is \'etale. Now
\cite[Theorem~1.3.(iii)]{6Op} and \cite[Theorem~5.10.(ii)]{Hol} imply that $g^!g_+\M$
is coadmissible, respectively holonomic if $g_+\M$ is.

We claim that the natural map
\[
g^!g_+\M=\O_Y\otimes_{g^{-1}\O_X} g^{-1}g_*\M\to \M,
\]
sending $b\otimes m$ to $bm$, admits locally a section, making $\M$ locally a direct summand.
 
For this, we can assume $Y=\Sp B$, $X=\Sp A$, where on the level of sections, the canonical $\Dcap_{Y}(Y)$-linear map
\[
    B \otimes_A \mathcal M(Y)\to \M(Y)
\] 
has the obvious section $m\mapsto 1\otimes m$.

Now apply Lemma~\ref{directsummand}
to deduce coadmissibility,
respectively holonomicity.
\end{proof}

Recall from \cite{Dcapthree} that a coadmissible $\Dcap_X$-module $\M$ is called \emph{weakly holonomic} if 
\begin{equation*}
	\mathcal{E}xt^i_{\Dcap_X}(\M, \Dcap_X)=0
\end{equation*}
for all $i\neq \dim X$.

As we have already seen in Lemma \ref{directsummand} that direct summands of coadmissible modules are coadmissible, it follows directly from functoriality that direct summands of weakly holonomic modules are weakly holonomic.

Similarly, we have the following partial analogue of Lemma \ref{fettest}:

\begin{lem}\label{fetwhol}
	Let $g: Y\to X$ be a finite \'etale morphism of smooth rigid analytic spaces. Let $\M$ be a coadmissible $\Dcap_Y$-module such that $g_+\M$ is weakly holonomic. Then $\M$ is weakly holonomic.
\end{lem}
\begin{proof}
	We have already seen above that $\M$ is a direct summand of $g^!g_+\M$. It thus suffices to observe that for any coadmissible $\Dcap_X$-module $\N$, we have
	\begin{equation*}
		\mathrm{R}g_*\mathrm{R}\mathcal{H}om_{\Dcap_Y}(g^!\N, \Dcap_Y)\cong \mathrm{R}\mathcal{H}om_{\Dcap_X}(\N, g_+\Dcap_X)\cong \mathrm{R}\mathcal{H}om_{\Dcap_X}(\N, \Dcap_Y)\otimes_{\Dcap_X}g_*\Dcap_Y 
	\end{equation*}
	e.g. by \cite[Theorem 6.27]{Global} and \cite[Theorem 3.12]{Hol}, as $g_+\Dcap_Y=g_*\Dcap_Y$ is a locally free $\Dcap_X$-module of finite rank. Since $\mathrm{R}\mathcal{H}om_{\Dcap_Y}(g^!\N, \Dcap_Y)$ is a right $\C$-complex and $g$ is an affinoid morphism, we deduce that
	\begin{equation*}
		g_*\mathcal{E}xt^i_{\Dcap_Y}(g^!\N, \Dcap_Y)\cong \mathcal{E}xt^i_{\Dcap_X}(\N, \Dcap_X)\otimes_{\Dcap_X}g_*\Dcap_Y,
	\end{equation*} 
	and taking $\N=g_+\M$ shows that $\M$ is weakly holonomic.
\end{proof}

\subsection{Snc divisors}

This section fills in some details in the proof of~\cite[Lemma 2.13]{MR4045974}.

\begin{prop}\label{cor:etale-boundary-charts}
Let $X$ be smooth, and let $D\subset X$ be an algebraic
\footnote{We use \emph{algebraic} snc divisors in order to apply Temkin's desingularisation theorem~\cite{TemkinDesingularization} following the method of~\cite[Sections~10.4--10.5]{Dcapthree} later on.}
snc divisor in the
sense of~\cite[\S9.2]{Dcapthree}. Then $X$ admits an admissible affinoid
covering $\{X_a\}_a$ such that, for every $a$, there are integers
$0\leq r_a\leq n_a$ and an étale morphism
\begin{equation*}
  g_a\colon X_a\to
  \Sp K\langle z_1,\dots,z_{n_a}\rangle
\end{equation*}
for which $D\cap X_a=g_a^{-1}(V(z_1\cdots z_{r_a}))$.
\end{prop}

\begin{proof}
The assertion is local on $X$. By \cite[Lemma 9.2]{Dcapthree}, we may therefore assume that $X=\Sp A$ admits an étale morphism
$f\colon X\to\mathbb B^n:=\Sp K\left\langle z_{1},\dots,z_{n}\right\rangle$
and that $D=\bigcup_{j=1}^qD_j$, $D_j=V(x_j)$, and $x_j=f^*z_j$
for some $q\leq n$.

We prove by induction on $q$ that $X$ has a finite affinoid covering by charts
\begin{equation}\label{eq:boundary-product-chart}
  U\cong
  Z\times\Sp K\left\langle z_{j}\colon j\in J\right\rangle,
  \qquad
  J\subseteq\{1,\dots,q\},
\end{equation}
where $Z$ admits an étale morphism to $\mathbb B^{n-\lvert J\rvert}$ and
\begin{equation*}
  D_j\cap U=
  \begin{cases}
    Z\times V(t_j),&j\in J,\\
    \emptyset,&j\notin J.
  \end{cases}
\end{equation*}
The case $q=0$ is immediate. Suppose that $q>0$. For any
$\epsilon\in K^\times$, the rational domains
\begin{equation*}
  X^+=\Sp A\langle\epsilon^{-1}x_q\rangle,
  \qquad
  X^-=\Sp A\langle\epsilon x_q^{-1}\rangle
\end{equation*}
form an admissible covering of $X$.
Put
\[
R_q
=
K\langle z_1,\ldots,\widehat{z_q},\ldots,z_n\rangle,
\qquad
S_q=R_q\langle z_q\rangle,
\qquad
Y_q=\operatorname{Sp}R_q.
\]
The morphism $f$ is induced by an étale homomorphism
$S_q\to A, z_j\mapsto x_j$.
Consequently,
\[
A/x_qA
\cong
A\widehat{\otimes}_{S_q}R_q,
\]
and hence $D_q=\operatorname{Sp}(A/x_qA)\to Y_q$
is étale.

By Kiehl's tubular-neighbourhood theorem
\cite[Satz~1.18]{Kiehl}, choosing $\epsilon$ sufficiently small, $X^+$
admits a finite affinoid covering by $U_\alpha=\Sp C_\alpha$ such that
\begin{equation*}
  C_\alpha
  \cong
  (C_\alpha/x_qC_\alpha)\langle t_q\rangle
\end{equation*}
as affinoid algebras.
Kiehl's proof even implies that this isomorphism
is $R_q$-linear.
The corresponding isomorphism
$U_\alpha
\cong
Z_\alpha\times\mathbb{B}^1$ for $Z_\alpha=\Sp (C_\alpha/x_qC_\alpha)$ is therefore an isomorphism over $Y_q$.
Finally, $Z_\alpha\to Y_q$ is étale because $Z_\alpha$ is an
admissible open subspace of $D_q$, while $D_q\to Y_q$ is étale.
Applying the induction hypothesis to
$\bigcup_{j<q}(D_j\cap Z_\alpha)$
and then taking the product with $\mathbb B^1$ gives the required charts
on $X^+$. 
Since $D_q\cap X^-=\emptyset$, the induction hypothesis applies
directly to $X^-$. This proves the claim.

For a chart \eqref{eq:boundary-product-chart}, taking the product of the
étale morphism $Z\to\mathbb B^{n-\lvert J\rvert}$
with the identity on $\mathbb B^{\lvert J\rvert}$ gives the required étale
morphism $U\to\mathbb B^n$.
\end{proof}

\begin{cor}\label{cor:finite-etale-boundary-charts}
In Proposition~\ref{cor:etale-boundary-charts}, the maps $g_a$ may be
chosen to be finite étale.
\end{cor}

\begin{proof}
Apply~\cite[Proposition~6.10]{Achinger} to each $Z_a$.
\end{proof}

\section{Residues and \texorpdfstring{$b$}{b}-functions}
\label{sec:log-res-b}

There are different ways of recording the monodromic behaviour of a connection around some (snc, say) boundary divisor. In \cite{Bitoun}, Bitoun and the first author described the relationship between coadmissible extensions and the roots of associated $b$-functions or Bernstein--Sato polynomials. In the study of logarithmic connections, similar information is expressed in terms of the eigenvalues of corresponding residue maps. In this section, we make precise how these two notions are linked. We expect that this is very well-known in more classical settings, even if it is not straightforward to locate a suitable reference. 

Let $A=K\langle z_1,\dots,z_n\rangle$, $X=\Sp A$, fix $m\leq n$, and set
$f=z_1\cdots z_m$. A \textbf{logarithmic connection} (or log-connection) along $V(f)$ is a finite free
$A$-module $E$ with an integrable connection
\[
  \nabla\colon E\longrightarrow
  E\otimes_A\left(\bigoplus_{i=1}^m A\frac{dz_i}{z_i}
  \oplus\bigoplus_{i=m+1}^n A\,dz_i\right).
\]

Thus $E$ is equipped with pairwise commuting operators describing the action of $D_i=z_i\partial_i$ for $i\leq m$ and $\partial_i$ for $i>m$, each satisfying a suitable version of the Leibniz rule.
For a given $i\leq m$, let
\[
  \overline D_i\colon E/z_iE\longrightarrow E/z_iE
\]
be the induced residue map.  By~\cite[Prop.--Def.~1.24]{MR2836060}, there exists a minimal polynomial $\mu_i(T)\in K[T]$ satisfied by $\overline{D_i}$. Write
$\Sigma_i\subset\overline K$ for its roots and
$\Sigma=\bigcup_{i=1}^m\Sigma_i$.

Fix an $A$-basis $v_1,\dots,v_d$ of $E$. By~\cite{MNM}, for each $a$
there are $P_a(s)\in\D(X)[s]$ and a nonzero $b_a(s)\in K[s]$ such that
\begin{equation}\label{eq:b-function}
  P_a(s)f^{-s}v_a=b_a(s)f^{-s-1}v_a.
\end{equation}
We take $b_a$ to be the monic polynomial of least degree with this property, the $b$-function or Bernstein--Sato polynomial of the element $v_a$.

\begin{thm}\label{thm:res-b}
Every root of $b_a(s)$ belongs to $\Sigma+\mathbb Z_{\geq0}$.
\end{thm}

\begin{remark}
  Classically, the relation between roots of $b$-functions $\lambda$ and eigenvalues $e^{2\pi i\lambda}$ of the associated monodromy operators under the Riemann--Hilbert correspondence was described by e.g. Malgrange and Kashiwara, see~\cite{Malgrange1975Bernstein,Malgrange1974Bernstein,Malgrange1983BernsteinSato,Kashiwara1983VanishingCycles}.
  We were, however, unable to find a purely algebraic proof in the literature which applies in our setting.
\end{remark}

Before giving a proof, we note the following immediate consequence.

\begin{cor}\label{cor:rationalresidues}
If the residue eigenvalues are rational, then so are the roots of every
$b_a(s)$.
\end{cor}

We first treat one variable.

\begin{lem}\label{thm:res-b-1var}
Let $B$ be an affinoid $K$-algebra, $A=B\langle z\rangle$, and
$D=z\partial_z$. Let $E$ be a \emph{log-connection relative to $B$}, that is
a finite free $A$-module equipped with an integrable connection
\begin{equation*}
    \nabla\colon E\to E\frac{dz}{z}.
\end{equation*}
Let $0\neq\mu(T)\in K[T]$ be a polynomial satisfied by the residue
map
\begin{equation*}
    \overline{D}\colon E/zE\to E/zE.
\end{equation*}
and let $\Sigma\subseteq\overline{K}$ be its set of roots.

For every $v\in E$ there are $P_v(s)\in\D_{A/B}[s]$ and
$q_v(s)\in K[s]\setminus\{0\}$ such that
\[
  P_v(s)z^{-s}v=q_v(s)z^{-s-1}v,
\]
and all roots of $q_v$ lie in $\Sigma+\mathbb Z_{\geq0}$.
\end{lem}

\begin{proof}
Let $F\subseteq E$ be the $A$-submodule generated by $v,Dv,D^2v,\dots$.  It is finitely generated over $A$ by Noetherianity
and $D$-stable. Its $z$-saturation
\[
  F^{\mathrm{sat}}=(F\otimes_AA[z^{-1}])\cap E
\]
is likewise finitely generated, so that $z^NF^{\mathrm{sat}}\subseteq F$ for some $N\geq0$.  Hence
\begin{equation}\label{eq:saturation}
  F\cap z^{N+1}F^{\mathrm{sat}}\subseteq zF.
\end{equation}

Put
\[
  q_v(T)=\prod_{r=0}^N\mu(T-r).
\]
On $z^rF^{\mathrm{sat}}/z^{r+1}F^{\mathrm{sat}}$, the operator $D$ is the
residue plus $r$. The filtration by these subquotients therefore gives
\[
  q_v(D)F^{\mathrm{sat}}\subseteq z^{N+1}F^{\mathrm{sat}}.
\]
Since $q_v(D)v\in F$, equation~\eqref{eq:saturation} yields
$q_v(D)v\in zF$.

For $w\in F$ one has
\[
  -\partial_z(z^{-s}w)=z^{-s-1}(sw-Dw)\in \D_{A/B}[s]\,z^{-s}F\cap z^{-s-1}F.
\]
Therefore, modulo $\D_{A/B}[s]\,z^{-s}F\cap z^{-s-1}F$, 
\[
z^{-s-1}sw-z^{-s-1}Dw\equiv 0
\]
for $w\in F$.
 
Consequently,
\[
  q_v(s)z^{-s-1}v\equiv z^{-s-1}q_v(D)v\equiv0
\]
by the above, so that $q_v(s)z^{-s-1}v\in \D_{A/B}[s]\,z^{-s}F$.

Moreover, we observe
\begin{equation*}
    \D_{A/B}[s]z^{-s}v=z^{-s}F.
\end{equation*}
In fact, to show $\subseteq$, let $x\in z^{-s}F$. Then
$x$ is a finite sum of elements $z^{-s}aD^iv\in E$, where
$a\in A$. The commutation relation $z^{-s}D^i=(D+s)^i z^{-s}$ yields
\[
    z^{-s}aD^iv
    =a(D+s)^i z^{-s}v
    =\left(\sum_{j=0}^{i}\binom{i}{j}s^{i-j}aD^j\right)z^{-s}v
    \in \D_{A/B}[s]z^{-s}v.
\]
This implies $x\in\D_{A/B}[s]z^{-s}v$, as desired. The proof of $\supseteq$ is similar.

Therefore, $q_v(s)z^{-s-1}v\in \D_{A/B}[s]\,z^{-s}F$, yielding the required $P_v(s)$. The roots of $q_v$ are the numbers
$\lambda+r$ with $\lambda\in\Sigma$ and $0\leq r\leq N$.
\end{proof}

\begin{proof}[Proof of Theorem~\ref{thm:res-b}]
For \(1\leq j\leq m\), write
\[
  A=B_j\langle z_j\rangle,
  \qquad
  B_j=K\langle z_1,\dots,\widehat z_j,\dots,z_n\rangle,
  \qquad
  h_j=f/z_j.
\]
Applying Lemma~\ref{thm:res-b-1var} relative to \(B_j\)
and $\mu_{j}$ playing the role of $\mu$, we obtain
\[
  P_{a,j}(s)\in\D_{A/B_j}[s],
  \qquad
  q_{a,j}(s)\in K[s]\setminus\{0\},
\]
such that
\[
  P_{a,j}(s)z_j^{-s}v_a
  =
  q_{a,j}(s)z_j^{-s-1}v_a,
\]
and every root of \(q_{a,j}\) belongs to
\(\Sigma_j+\mathbb Z_{\geq0}\).

Since \(P_{a,j}(s)\) differentiates only in the \(z_j\)-direction, it commutes with the formal factor \(h_j^{-s}\).
Multiplying the preceding identity by \(h_j^{-s}\) therefore gives
\begin{equation}\label{eq:directional-b-relation}
  P_{a,j}(s)f^{-s}v_a
  =
  q_{a,j}(s)z_j^{-1}f^{-s}v_a.
\end{equation}
Moreover, if \(i\neq j\), every element of
\(\D_{A/B_j}[s]\) commutes, after localization at \(f\), with
multiplication by \(z_i^{-1}\). Indeed, such an operator is generated
by multiplication by elements of \(A\) and differentiation in the
\(z_j\)-direction, while \(\partial_j(z_i^{-1})=0\).

Set \(u=f^{-s}v_a\). We claim that, for \(1\leq r\leq m\),
\begin{equation}\label{eq:successive-directional-relations}
  P_{a,r}(s)\cdots P_{a,1}(s)u
  =
  \left(\prod_{j=1}^r q_{a,j}(s)\right)
  \left(\prod_{j=1}^r z_j^{-1}\right)u.
\end{equation}
For \(r=1\), this is~\eqref{eq:directional-b-relation}. If it holds
for some \(r<m\), then the preceding commutation property and
\eqref{eq:directional-b-relation} for \(j=r+1\) give
\[
\begin{aligned}
  P_{a,r+1}(s)\cdots P_{a,1}(s)u
  &=
  \left(\prod_{j=1}^r q_{a,j}(s)\right)
  P_{a,r+1}(s)
  \left(\prod_{j=1}^r z_j^{-1}\right)u \\
  &=
  \left(\prod_{j=1}^r q_{a,j}(s)\right)
  \left(\prod_{j=1}^r z_j^{-1}\right)
  P_{a,r+1}(s)u \\
  &=
  \left(\prod_{j=1}^{r+1}q_{a,j}(s)\right)
  \left(\prod_{j=1}^{r+1}z_j^{-1}\right)u.
\end{aligned}
\]
This proves the claim by induction.

Taking \(r=m\) in~\eqref{eq:successive-directional-relations} and using
\(\prod_{j=1}^m z_j^{-1}=f^{-1}\), we obtain
\[
  P_{a,m}(s)\cdots P_{a,1}(s)f^{-s}v_a
  =
  \left(\prod_{j=1}^m q_{a,j}(s)\right)f^{-s-1}v_a.
\]
Thus the $b$-function \(b_a(s)\) divides \(\prod_{j=1}^m q_{a,j}(s)\).
Every root of \(b_a(s)\) is therefore
a root of some \(q_{a,j}(s)\), and hence belongs to
\[
  \bigcup_{j=1}^m
  \left(\Sigma_j+\mathbb Z_{\geq0}\right)
  =
  \Sigma+\mathbb Z_{\geq0}. \qedhere
\]
\end{proof}

\begin{cor}\label{postypeext}
	Let $E$ be a logarithmic connection over $A$ along $V(f)$ such that all residue eigenvalues are of positive type (see \cite[Definition 13.1.1]{KedlayapDE}). If $\E_U$ denotes the associated integrable connection on $U=X\setminus V(f)$ and $j: U\to X$, then $j_+\E_U=j_*\E_U$ is a coadmissible, weakly holonomic $\Dcap_X$-module.
\end{cor}
\begin{proof}
	Recall from \cite[Definition 13.1.1]{KedlayapDE} that $\lambda\in \overline{K}$ is of positive type if the formal power series
	\begin{equation*}
		\sum_{n\geq 0, n\neq \lambda} \frac{x^n}{\lambda-n}
	\end{equation*}
	has positive radius of convergence. In particular, if $\lambda$ is of positive type, then so is $\lambda+n$ for any $n\in \mathbb{Z}_{\geq 0}$. 
	
	Hence Theorem \ref{thm:res-b} implies that $E$ admits a basis $v_1, \hdots, v_d$ whose associated $b$-functions have all their roots of positive type. It follows from \cite[Theorem 1.2]{Bitoun} that $j_+\E_U=j_*\E_U$ is a coadmissible $\Dcap_X$-module.
	
	Moreover, \cite[Theorem 1.2]{Bitoun} yields in this case that
	\begin{equation*}
		j_*\E_U\cong \Dcap_X\otimes_{\D_X} (\O_X(*D)\otimes_AE),
	\end{equation*}
	so that the weak holonomicity follows from the same argument as in \cite[Proposition 4.1.4]{MNM} and \cite[Proposition 7.2]{Dcapthree}.
\end{proof}

\section{Logarithmic connections with nilpotent residues}
\label{sec:nilpotent-residues}

Suppose $j:U\to X$ is a Zariski-open immersion of rigid analytic $K$-varieties and $\E$ is an integrable connection on $U$. The most straightforward way to ensure that $j_+\E$ is a holonomic complex is to require that there exists some integrable connection $\M$ on $X$ such that $\M|_U\cong \E$, where we can invoke \cite[Theorem 5.10.(iv), Theorem 5.12]{Hol}, as $j_+\E\cong j_+j^!\M$ in this case. More generally, we can extend this argument to allow for those $\E$ which admit a filtration such that each successive quotient has this property. 

In this section, we show that any log-connection with nilpotent residues admits such a filtration locally. We remark that the question of extending log-connections with nilpotent residues to connections in this way is usually related to the phenomenon of unipotence, see \cite{MR2360314} for a similar discussion. Our approach here is very similar, but we emphasize that we only describe our log-connections locally as the restrictions of \emph{some} integrable connection, not necessarily as the structure sheaf. Still, our methods are very close to those in \cite{MR2360314}.   

Following this perspective, the key point of this section can also be formulated as follows: by standard arguments, each point admits an open neighbourhood where our given log-connection with nilpotent residues is unipotent, in the sense that it admits a filtration such that each graded piece trivializes (see e.g.~\cite[Lemma 3.2.17]{MR2360314}). However, as is common in this setting, it is not clear that these open neighbourhoods form an \emph{admissible} covering. By weakening our desired notion of extendability, we are able to control the radius of convergence of suitable base-change matrices to arrive at an admissible covering. In fact, it is possible to derive explicit bounds for the radii analogously to \cite[Proposition 18.1.1]{KedlayapDE}, but we do not spell this out here.  

\subsection{A filtration criterion for holonomicity}

Fix $n\geq m$, let
$I^+\sqcup I^-=\{1,\dots,m\}$, and choose radii $0<r_i\leq1$ for $i\in P$.  Set
\begin{align*}
 A^+&:=K\langle r_i^{-1}z_i:i\in I^+\rangle, \\
 A^-&:=K\langle z_i\ \colon i\in I^-, \ z_{m+1},\dots,z_{n} \rangle, \\
 A&:=A^+\widehat\otimes_KA^-.
\end{align*}

Throughout this subsection, we will have the following assumptions in place:

Let $E$ be finite free of rank $d$ over $A$, equipped with commuting
logarithmic derivations $D_i=z_i\partial_i$ for $i\in I^+$. Suppose that, in
some basis $v_1,\dots,v_d$, their matrices have entries in $A^-$ and are
nilpotent. We remark that in this case, the residue $\overline{D_i}$ is nilpotent, as $D_i$ commutes with $z_j$ for each $j\neq i$.

Write $E^-=\sum_jA^-v_j$ and
\[
 A_{\widehat i}
 =A^-\widehat\otimes_K
 K\langle r_j^{-1}z_j:j\in I^+\setminus\{i\}\rangle.
\]

\begin{lem}\label{kernelcontrol}
\leavevmode
\begin{itemize}
\item[(i)] For $i\in I^+$ and $l\geq1$,
\[
 \ker(D_i^l)
 =A_{\widehat i}\otimes_{A^-}\ker\bigl(D_i|_{E^-}^l\bigr).
\]
\item[(ii)] For every $l\geq 1$
\[
 \bigcap_{i\in I^+}\ker(D_i^l)
 =\bigcap_{i\in I^+}\ker\bigl(D_i|_{E^-}^l\bigr)
\]
\item[(iii)] In particular,
\[
 \bigcap_{i\in I^+}\ker(D_i^d)=E^-.
\]
\item[(iv)] Every $A$-linear endomorphism of $E$ commuting with the $D_i$ preserves
$E^-$.
\end{itemize}
\end{lem}

\begin{proof}
(i) For $i\in I^+$, set
$N_i:=D_i|_{E^-}$.
The matrix of $\overline{D_i}$ is exactly the matrix of $N_i$ relative to the basis $v_1, \hdots, v_d$, as we assume each entry to be in $A^-$. Hence $N_i$
is nilpotent. Furthermore, if $L$ denotes the fraction field of $A^-$,
then $N_i$ is a nilpotent endomorphism of the $d$-dimensional $L$-vector
space $L\otimes_{A^-}E^-$.
Cayley--Hamilton therefore gives
\begin{equation}\label{eq:nilpotence-Ni}
  N_i^d=0.
\end{equation}

Fix $l\geq1$. Every $e\in E$ has a unique expansion
\begin{equation*}
  e=\sum_{\nu\geq0}z_i^\nu e_\nu,
  \qquad
  e_\nu\in A_{\widehat i}\otimes_{A^-}E^-.
\end{equation*}
The operator $D_i$ is $A_{\widehat i}$-linear and restricts to $N_i$ on
$E^-$. The Leibniz rule consequently gives
\begin{equation*}
  D_i^l(e)
  =
  \sum_{\nu\geq0}
  z_i^\nu\bigl(\nu\operatorname{id}+N_i\bigr)^l e_\nu.
\end{equation*}
For every $\nu\geq1$, the endomorphism
$\nu\operatorname{id}+N_i$ is invertible. Indeed, by
\eqref{eq:nilpotence-Ni}, its inverse is
\begin{equation*}
  \bigl(\nu\operatorname{id}+N_i\bigr)^{-1}
  =
  \nu^{-1}
  \sum_{k=0}^{d-1}(-\nu^{-1}N_i)^k.
\end{equation*}
It follows from the uniqueness of the $z_i$-expansion that, if
$D_i^l(e)=0$, then
\begin{equation*}
  e_\nu=0\quad(\nu\geq1),
  \qquad
  N_i^l(e_0)=0.
\end{equation*}
Expanding $e_0$ in the remaining variables $z_j$, $j\in I^+\setminus\{i\}$,
shows coefficientwise that
\begin{equation*}
  \ker\bigl(N_i^l:
  A_{\widehat i}\otimes_{A^-}E^-
  \to
  A_{\widehat i}\otimes_{A^-}E^-\bigr)
  =
  A_{\widehat i}\otimes_{A^-}\ker(N_i^l).
\end{equation*}
We therefore obtain
\begin{equation*}
  \ker(D_i^l)
  =
  A_{\widehat i}\otimes_{A^-}
  \ker\bigl(D_i|_{E^-}^l\bigr).
\end{equation*}

(ii) Now write an element of $E$ in the form
\begin{equation*}
  e=\sum_{\alpha\in\mathbb N^{I^+}}z^\alpha e_\alpha,
  \qquad e_\alpha\in E^-.
\end{equation*}
If $e\in\ker(D_i^l)$, the formula just proved shows that $e$ is independent
of $z_i$; i.e., $e_\alpha=0$ whenever $\alpha_i>0$. Thus an element
belonging to $\ker(D_i^l)$ for every $i\in I^+$ is independent of all the
variables $z_i$, $i\in I^+$, and hence lies in $E^-$. On $E^-$, the operator
$D_i$ agrees with $N_i$, so
\begin{equation*}
  \bigcap_{i\in I^+}\ker(D_i^l)
  =
  \bigcap_{i\in I^+}
  \ker\bigl(D_i|_{E^-}^l\bigr).
\end{equation*}

(iii) Taking $l=d$ and using \eqref{eq:nilpotence-Ni} gives $\bigcap_{i\in I^+}\ker(D_i^d)=E^-$.

(iv) Finally, let $\varphi\colon E\to E$ be an $A$-linear endomorphism commuting
with every $D_i$. For $e\in\ker(D_i^d)$, we have
\begin{equation*}
  D_i^d\varphi(e)=\varphi D_i^d(e)=0.
\end{equation*}
Thus $\varphi$ preserves each $\ker(D_i^d)$ and consequently preserves their
intersection $E^-$.
\end{proof}

Let $X'=\Sp B$ be an affinoid subdomain of $\Sp A$ on which every $z_i$,
$i\in I^-$, is invertible. Put
\[
 D^+=V\left(\prod_{i\in I^+}z_i\right)\cap X',
 \qquad U'=X'\setminus D^+,
\]
and let $\cE_{U'}$ be the connection induced by $E$ on $U'$.

	\begin{prop}\label{extensionofextendables}
		There exists a finite filtration by sub-connections
		\begin{equation*}
			0\subsetneq \cE_1\subsetneq \cE_2\subsetneq \hdots \subsetneq \cE_k=\cE_{U'}
		\end{equation*}
		such that for each $s$, there exists some integrable connection $\N_s$ on $X'$ with
		\begin{equation*}
			\N_s|_{U'}\cong \cE_s/\cE_{s-1}.
		\end{equation*}
	\end{prop}
	
\begin{proof}
The proof strategy (which may be considered a relative trivialization) is a common one in the study of unipotent connections: within the logarithmic connection $B\otimes_A E$, the intersection
\begin{equation*}
    \bigcap_{i\in I^+}\ker(D_i)\subseteq E^-
\end{equation*}
spans a $B$-submodule on which $D_i$, and hence $\partial_i$, acts trivially for $i\in I^+$, and which is stable under $\partial_i$ for each $i\in I^-$ (it is stable under $D_i=z_i\partial_i$ by commutation relations, and $z_i\in B^\times$ by assumption) as well as under $\partial_i$ for $i>m$ by the commutation relations. This produces the first step in the filtration, which we can then obtain inductively. 

Slightly more explicitly: For $i\in I^+$, put $N_i=D_i|_{E^-}$. For $s\geq1$, define
\[
 F_s=
 \bigcap_{(i_1,\dots,i_s)\in {I^+}^s}
 \ker(N_{i_1}\cdots N_{i_s}),
 \qquad F_0=0,
\]
which are $A^-$-submodules of $E^-$.

The operators $N_i$ commute and satisfy $N_i^d=0$ by
Lemma~\ref{kernelcontrol}(i). Hence
\[
 0=F_0\subset F_1\subset\cdots\subset F_e=E^-
\]
for $e=\lvert I^+\rvert(d-1)+1$, and
\[
 N_i(F_s)\subset F_{s-1}
 \qquad (i\in I^+).
\]

We claim that the remaining connection operators preserve every $F_s$.
Indeed, let $T$ be $D_j$ for $j\in I^-$, or $\partial_j$ for $j>m$.
Integrability gives $[T,D_i]=0$ for every $i\in I^+$. Since
\[
 E^-=\bigcap_{i\in I^+}\ker(D_i^d)
\]
by Lemma~\ref{kernelcontrol}(iii), it follows that $T(E^-)\subset E^-$. The
commutation relations then imply $T(F_s)\subset F_s$.

The morphism $A^-\to B$ is flat. Thus
$E'_s:=B\otimes_{A^-}F_s$
is a submodule of $B\otimes_{A^-}E^-\cong B\otimes_AE$.
Set
\[
 \cE_s:=\O_{U'}\otimes_BE'_s.
\]
The preceding observations show that the remaining connection operators
preserve $\cE_s$. Moreover, for $i\in I^+$ we have
$D_i(F_s)\subset F_s$, and $z_i$ is invertible on $U'$. Hence
$\partial_i=z_i^{-1}D_i$ also preserves $\cE_s$. Therefore $\cE_s$ is a
sub-connection of $\cE_{U'}$.

Put $G_s=F_s/F_{s-1}$.  Each $N_i$, $i\in I^+$, acts trivially on $G_s$.
Let $\N_s$ be the coherent $\O_{X'}$-module associated with
$B\otimes_{A^-}G_s$. The connection operators in the remaining directions
descend to this module, while for $i\in I^+$ we define
\[
 \nabla_{\partial_i}(b\otimes\overline v)
   =\partial_i(b)\otimes\overline v.
\]
These operators commute, so $\N_s$ carries an integrable connection on
$X'$. In particular, it is locally free. Since $D_i$ acts trivially on
$G_s$, its restriction to $U'$ agrees with the quotient connection, and
therefore
\[
 \N_s|_{U'}\cong\cE_s/\cE_{s-1}.
\]
After omitting repetitions from the filtration $(\cE_s)_s$, we obtain the
required finite filtration.
\end{proof}

\begin{cor}\label{extendablehol}
Let $j\colon U'\hookrightarrow X'$ be the natural open immersion.  Then
$j_+\cE_{U'}$ is a coadmissible, weakly holonomic, and holonomic $\Dcap_{X'}$-module.
\end{cor}

\begin{proof}
Let
\[
 0=\cE_0\subset\cE_1\subset\cdots\subset\cE_k=\cE_{U'}
\]
and $\N_s$ be as in Proposition~\ref{extensionofextendables}. By
\cite[Theorem~5.12]{Hol}, each $\N_s$ is holonomic. Since
$\N_s|_{U'}=j^!\N_s$, it follows from
\cite[Theorem~5.10.(iv)]{Hol} that
\[
 j_+(\cE_s/\cE_{s-1})
 \cong j_+(\N_s|_{U'})
\]
is holonomic.

Applying $j_+$ to the short exact sequences defining the filtration gives
distinguished triangles
\[
 j_+\cE_{s-1}\to j_+\cE_s
 \to j_+(\N_s|_{U'})
 \to j_+\cE_{s-1}[1].
\]
Holonomic complexes form a triangulated subcategory by
\cite[Lemma~5.2]{Hol}; induction on $s$ therefore shows that
$j_+\cE_{U'}$ is a holonomic complex (and in particular, a $\C$-complex).

Finally, for every affinoid $V\subset X'$, the complement
\[
 V\cap U'=V\setminus V\Bigl(\prod_{i\in P}z_i\Bigr)
\]
is quasi-Stein. Kiehl's Theorem~B implies that higher direct images of
coherent $\O_{U'}$-modules vanish. Hence $j_+\cE_{U'}=Rj_*\cE_{U'}$ is
concentrated in degree zero and is therefore a holonomic
$\Dcap_{X'}$-module (and in particular, a coadmissible $\Dcap_{X'}$-module).

The weak holonomicity follows in the same way from \cite[Lemma 10.5]{Dcapthree}.
\end{proof}

\subsection{An application of the criterion}

Let $A=K\langle z_1,\dots,z_n\rangle$, let $m\leq n$, and let $E$ be a finite
free logarithmic connection along $V(z_1\cdots z_m)$, with nilpotent
residues. For $\sigma=(\sigma_1, \hdots, \sigma_m)\in\{\pm1\}^m$, set
\[
 I_\sigma^{+}=\{i:\sigma_i=1\},
 \qquad I_\sigma^{-}=\{i:\sigma_i=-1\}.
\]

\begin{prop}\label{locallyextendable}
There are radii $0<r_i<1$ with associated affinoids
\[
 X_\sigma=\Sp A_\sigma,
 \qquad
 A_\sigma=
 K\langle r_i^{-1}z_i,\ z_j,r_jz_j^{-1},\ z_{m+1},\dots,z_n:
 i\in I_\sigma^{+}, j\in I_\sigma^{-}\rangle,
\]
such that on each $X_\sigma$, the restriction of the connection
satisfies the hypotheses of Proposition~\ref{extensionofextendables} with
$I^+=I_\sigma^{+}$ and $I^-=I_\sigma^{-}$.
\end{prop}

\begin{proof}
Fix $\sigma\in \{\pm 1\}^m$, write $I^+=I_\sigma^{+}$ and $I^-=I_\sigma^{-}$, and put
\[
 R=K\langle z_j\colon j\in I^-, \ z_{m+1},\dots,z_n\rangle.
\]
We prove by induction on $\lvert I^+\rvert$ that, after restricting the variables
$z_i$, $i\in I^+$, to sufficiently small closed discs, there is a basis in
which the matrix of each $D_i=z_i\partial_i$, $i\in I^+$, belongs to
$\mathrm{Mat}_d(R)$ and is nilpotent. The assertion is clear if $I^+$ is
empty.

Choose $k\in I^+$ and consider 
\[
 \overline E=E/z_kE.
\]
The operators $D_i$, $i\in I^+\setminus\{k\}$, descend to commuting
logarithmic derivations on $\overline E$, and their residues remain
nilpotent.  By induction, after shrinking the discs in the variables
$z_i$, $i\in I^+\setminus\{k\}$, there is a basis of $\overline E$ in which
the matrices of these operators belong to $\mathrm{Mat}_d(R)$ and are
nilpotent. Write
\[
 C=R\langle r_i^{-1}z_i:i\in I^+\setminus\{k\}\rangle
\]
for the resulting coefficient ring. Lifting the corresponding
change-of-basis matrix independently of $z_k$, we obtain a basis of $E$
over $C\langle z_k\rangle$ whose reduction modulo $z_k$ is the chosen
basis of $\overline E$.

Let
\[
 N(z_k)=\sum_{\nu\geq0}N_\nu z_k^\nu
 \in\mathrm{Mat}_d(C\langle z_k\rangle)
\]
be the matrix of $D_k$ in this basis, and set $N_0=N(0)$.  Modulo $z_k$,
the operator $D_k$ is $C$-linear and commutes with every $D_i$,
$i\in I^+\setminus\{k\}$. Lemma~\ref{kernelcontrol}(iv), applied to
$\overline E$, therefore shows that $D_k$ preserves the $R$-span of the
chosen basis. Equivalently,
\[
 N_0\in\mathrm{Mat}_d(R).
\]
Moreover, $N_0$ represents the residue of the connection along $z_k=0$,
so it is nilpotent.

We now construct a change of basis making the matrix in the
$k$-direction equal to $N_0$. For this, seek
\[
 T=1+\sum_{\nu>0}T_\nu z_k^\nu\in \mathrm{Mat}_d(R[[z_k]])
\]
satisfying
\begin{equation}\label{eq:gauge}
 N(z_k)T+z_k\partial_k(T)=TN_0.
\end{equation}
Put
\begin{align*}
\Phi: &\mathrm{Mat}_d(R)\to \mathrm{Mat}_d(R)\\
 &C\mapsto N_0C-CN_0.
\end{align*}
Comparing the coefficient of $z_k^\nu$ in~\eqref{eq:gauge} gives
\begin{equation}\label{eq:gauge-recursion}
 (\nu+\Phi)T_\nu
 =
 -\sum_{\mu=1}^{\nu}N_\mu T_{\nu-\mu}.
\end{equation}
Choose $h$ such that $N_0^h=0$ and put $q=2h-1$. Then
$\Phi^q=0$, since every term in the expansion of
$\Phi^q(C)$ contains either $N_0^h$ on the left or on the right. Hence,
for every $\nu>0$,
\[
 (\nu+\Phi)^{-1}
 =
 \sum_{a=0}^{q-1}(-1)^a\nu^{-a-1}\Phi^a.
\]
Thus~\eqref{eq:gauge-recursion} determines the matrices $T_\nu$
uniquely.

It remains to verify convergence. Choose a Banach norm on $C$ and the
associated matrix norm. Since $N(z_k)$ is analytic, the coefficients
$N_\mu$ are bounded. The preceding formula therefore gives a constant
$c_0>0$, independent of $\nu$, such that
\[
 \|T_\nu\|
 \leq
 c_0|\nu|^{-q}
 \max_{0\leq j<\nu}\|T_j\|.
\]
If
\[
 u_\nu=\max_{0\leq j\leq\nu}\|T_j\|,
\]
then, after enlarging $c_0$ if necessary,
\[
 u_\nu\leq c_0|\nu|^{-q}u_{\nu-1},
 \qquad
 u_\nu\leq c_0^\nu|\nu!|^{-q}.
\]
Since $\lvert\nu!\rvert^{-1/\nu}$ is bounded, the coefficients $T_\nu$
have at most exponential growth. Consequently, $T$ converges after
restricting $z_k$ to a disc of some positive radius. By shrinking this
radius once more, we may arrange that $\|T-1\|<1$; then $T$ is invertible.

In the resulting basis the matrix of $D_k$ is $N_0$.  Let
\[
 A_i(z_k)=\sum_{\nu\geq0}A_{i,\nu}z_k^\nu
\]
be the matrix of $D_i$ for $i\in P\setminus\{k\}$. Integrability gives
\[
 z_k\partial_k(A_i)-z_i\partial_i(N_0)+[N_0,A_i]=0.
\]
Since $N_0$ has entries in $R$, we have $z_i\partial_i(N_0)=0$. Hence
\[
 (\nu+\Phi)A_{i,\nu}=0
 \qquad(\nu>0).
\]
The operator $\nu+\Phi$ is invertible, so $A_{i,\nu}=0$ for every
$\nu>0$.  Moreover, $T\equiv1\pmod{z_k}$, and therefore
$A_{i,0}$ is the matrix furnished by the inductive hypothesis. Thus
\[
 A_i(z_k)=A_{i,0}\in\mathrm{Mat}_d(R),
\]
and this matrix is nilpotent. This completes the induction.

For each $\sigma$, the preceding argument produces positive radii for the
variables indexed by $I_\sigma^+$. Given $i\leq m$, we can then find a common
radius $r_i$ which works for all $\sigma$ with $i\in I_\sigma^+$.
With
\[
 B_\sigma
 =
 K\langle r_i^{-1}z_i:i\in I_\sigma^+\rangle
 \widehat\otimes_K
 K\langle z_j\colon j\in I_\sigma^-, \ z_{m+1},\dots,z_n\rangle,
\]
we have
\[
 A_\sigma
 =
 B_\sigma\langle r_jz_j^{-1}:j\in I_\sigma^-\rangle.
\]
The basis constructed over $B_\sigma$ remains a basis after this
localization, and the matrices of $D_i$, $i\in P_\sigma$, remain nilpotent
with entries in
\[
 B_\sigma^-
 =
 K\langle z_j \colon j\in I_\sigma^-, \ z_{m+1},\dots,z_n\rangle.
\]
Thus the hypotheses of Proposition~\ref{extensionofextendables} hold on
$X_\sigma$.

Finally, for each $i\leq m$, the two rational subdomains
\[
 \{|z_i|\leq r_i\}
 \qquad\text{and}\qquad
 \{|z_i|\geq r_i\}
\]
cover the unit disc in the $z_i$-coordinate. Taking their products over
$i=1,\dots,m$ gives precisely the affinoids $X_\sigma$, so these form a
finite affinoid covering of $\Sp A$.
\end{proof}

Now suppose that $X=\Sp B$ is a smooth affinoid $K$-variety with finite \'etale chart 
\begin{equation*}
  g: X\to \mathbb{B}^n=\Sp A=\Sp K\langle z_1, \hdots, z_n\rangle.   
\end{equation*}
Let $H=V(z_1\hdots z_m)$ and $V=\mathbb{B}^n\setminus H$, and $D=g^{-1}V(z_1\hdots z_m)$, $U=g^{-1}V=X\setminus D$, with the induced maps
	\begin{equation*}
		\begin{xy}
			\xymatrix{U\ar[r]^{j}\ar[d]_{g_U}& X\ar[d]^g\\
				V\ar[r]_{\ell}& \mathbb{B}^n}
		\end{xy} 
	\end{equation*}
Write $x_i$ for the image of $z_i$ in $B$.
	
A log-connection on $X$ relative to $D$ is a (locally) free $\O_X$-module with integrable connection $\nabla: \E\to \E\otimes_{\O_X} \Omega^1(\mathrm{log}D)$. Consequently, $\E$ is equipped with pairwise commuting operators $x_i\partial_i$ for $i\leq m$ and $\partial_i$ for $m<i\leq n$, satisfying the Leibniz rule, where in a slight abuse of notation $\partial_i$ is the element of the tangent sheaf corresponding to $1\otimes \partial_i=1\otimes\frac{\partial}{\partial z_i}$ under the identification $\mathrm{Der}_K(B)\cong B\otimes_A \mathrm{Der}_K(A)$.
	
It follows that $g_*\E$ is a vector bundle on $\mathbb{B}^n$, and hence free by ~\cite[\S1, Satz~1]{LutkebohmertVectorBundles}, with the log-connection on $X$ defining the structure of a log-connection on $g_*\E$ over the polydisk (or, in our earlier language, on the $A$-module $E:=g_*\E(\mathbb{B}^n)=E(X)$). As $E/z_iE\cong E/x_iE$ by definition, we see directly that the residue map $\overline{x_i\partial_i}: E/x_iE\to E/x_iE$ agrees with the residue map for the log-connection $g_*\E$.
	
We say that $\E$ is a log-connection with nilpotent residues if each of these residue maps is a nilpotent map -- by the paragraph above, this is equivalent to requiring $g_*\E$ to be a log-connection with nilpotent residues over the polydisk.
	
More generally, if $X$ is any smooth rigid analytic $K$-variety and $D\subset X$ is an algebraic snc divisor, we have seen in Corollary~\ref{cor:finite-etale-boundary-charts} that it is locally of the form discussed above. It is thus straightforward to define log-connections with nilpotent residues on this level of generality by arguing locally, noting that the above is independent of the choice of coordinate chart.

\begin{cor}\label{cor:nilpotent-log-holonomic}
Let $X$ be smooth, let $D\subset X$ be an algebraic snc divisor, and let $\cE$ be a
logarithmic connection on $(X,D)$ with nilpotent residues.  If
$j\colon U=X\setminus D\hookrightarrow X$, then $j_+(\cE|_U)=j_*(\cE|_U)$ is coadmissible, weakly holonomic, and holonomic.
\end{cor}

\begin{proof}
First, note that $j_+(\E|_U)=Rj_*(\E|_U)$ is indeed concentrated in degree zero: if $W$ is any affinoid subdomain, then $W\cap U$ is quasi-Stein, so that $R^qj_{W,*}(\E|_{W\cap U})=0$ for $q>0$ and $j_W: W\cap U\to W$ by Kiehl's Theorem B. Thus $j_+(\E|_U)=j_*(\E|_U)$, and it remains to establish the three finiteness conditions of coadmissibility, weak holonomicity, and holonomicity.

The assertion is local on $X$, so that we can assume that $X$ admits a finite étale boundary chart
\[
 g\colon X\to
 \mathbb{B}^n=\Sp K\langle z_1,\dots,z_n\rangle
\]
as in Corollary~\ref{cor:finite-etale-boundary-charts}, so that
$D\cap X=g^{-1}(H)$, with the same notation as above.

Set $\mathcal{F}=g_*\E$. By the discussion above, $\mathcal{F}$ is a log-connection on the polydisk relative to $H$, with nilpotent residues.

Hence, Proposition~\ref{locallyextendable} and
Corollary~\ref{extendablehol} imply that $\ell_+(\mathcal{F}|_V)$ is coadmissible, weakly holonomic, and holonomic. Since $g\circ j=\ell\circ g_U$ and $g_{U, +}(\E|_U)\cong \mathcal{F}|_V$, we obtain
\begin{equation*}
    g_+j_+(\E|_U)\cong \ell_+(\mathcal{F}|_V).
\end{equation*}

Lemmas~\ref{fettest} and~\ref{fetwhol} therefore imply that
$j_{+}(\cE|_{U})$ is also coadmissible, weakly holonomic, and holonomic.

\end{proof}

\section{Proof of the main results}
\label{sec:main-proofs}

Now we work again in a more general setup.
Let $X$ be a smooth rigid analytic $K$-variety, $Z\subseteq X$ a Zariski closed subvariety
and $j\colon U:=X\setminus Z\to X$ the embedding.

In the following, we recall Liu--Zhu's (arithmetic) Riemann-Hilbert functor~\cite[\S3.2]{LiuZhuRH2017}
\footnote{Liu--Zhu denote it by $D_{\dR}^{0}$.}.
We consider an integral variant so that we can apply~\cite{MR4592580} later on.

The pro-étale site $X_{\proet}$ of $X$ consists of formal limits
$\text{``}\varprojlim\text{"}_{i\in I}U_{i}$ of $U_{i}\in X_{\et}$
such that $I$ is a cofiltered category and that the
transition maps $U_{j}\to U_{i}$ are finite étale and surjective.
A collection $\{f_{i}\colon U_{i}\to U\}_{i}$ in $X_{\proet}$ is
a covering if it is a pointwise covering, and a set-theoretic condition
is satisfied which we ignore for simplicity.
There is a canonical projection of sites $\nu\colon X_{\proet}\to X$.

Let $\widehat{\ZZ}_{p}:=\varprojlim \ZZ/p^{n}$ as sheaves on $X_{\proet}$.
For every $\ZZ_{p}$-étale local system $\bL$ on $X$, we obtaine the sheaf
$\widehat{\bL}:=\varprojlim\nu^{-1}\bL/p^{n}$ of $\widehat{\ZZ_{p}}$-modules on $X_{\proet}$.

Furthermore, we have the de Rham period structure sheaf $\OB_{\dR}$ on $X_{\proet}$.
For the precise definition, we refer the reader to~\cite{Sch13pAdicHodgeErratum}
and simply recall that it carries additional structure:
\begin{itemize}
  \item There is a canonical morphism $\alpha\colon\nu^{-1}\mathcal{O}\to\OB_{\dR}$ of sheaves of rings.
  \item $\OB_{\dR}$ carries a differential which commutes with the map $\alpha$.
  \item There is a decreasing filtration on $\OB_{\dR}$ called the \emph{Hodge filtration}.
\end{itemize}

We can now define the Riemann-Hilbert functor. If $\bL$ is an étale $\ZZ_{p}$-local system
on $X$, then we associate to it the filtered $\mathcal{O}$-module
\begin{equation*}
D_{\dR}\left(\bL\right):=\nu_{*}\left(\widehat{\bL}\otimes_{\widehat{\ZZ}_{p}}\OB_{\dR}\right)
\end{equation*}
with integrable connection satisfying Griffiths transversality.

We say that an étale $\mathbb{Z}_{p}$-local system $\bL$ is \emph{de Rham} if
$\widehat{\bL}$ is de Rham in the sense
of~\cite[Definition 8.3]{Sch13pAdicHodge}.

We now explain how the theory of the previous sections can be applied to connections coming from a de Rham local system.

\begin{lem}\label{lem:DdR-comp-directimages}
    Fix a smooth proper map $f\colon Y\to U$ of
    smooth rigid analytic $K$-varieties, and
    any de Rham $\mathbb{Z}_p$-local system $\bL$
    on $Y_{\et}$.
    Then the higher pushforwards
    $R^{q}f_{\et,*}\bL$ are de Rham
    $\mathbb Z_p$-local systems on $U_{\et}$,
    and
    the canonical morphisms
    \[
    D_{\dR}\left(R^{q}f_{\et,*}\bL\right)
    \stackrel{\cong}{\longrightarrow}
    R^{q}f_{+}D_{\dR}\left(\bL\right)
    \]
    are isomorphisms of filtered vector bundles with integrable connection
    for every $q\geq0$.
\end{lem}

\begin{proof}
    \cite[Corollary 6.3.5]{MR4592580} shows that
    $R^{q}f_{\et,*}\bL$ is a $\mathbb{Z}_p$-local system, and
    \cite[Proposition 5.2.1 and Lemma 6.3.3]{MR4592580} give the
    canonical identification
    \[
    R^{q}f_{\proet,*}\widehat{\bL}
    \cong
    \widehat{R^{q}f_{\et,*}\bL}.
    \]
    Hence $R^{q}f_{\proet,*}\widehat{\bL}$
    is a lisse $\widehat{\mathbb{Z}}_{p}$-local system
    on $U_{\proet}$.
    Thus \cite[Theorem 8.8(ii)]{Sch13pAdicHodge}
    applies, which in our notation says precisely
    $R^qf_+D_{\dR}(\bL)\cong D_{\dR}(R^qf_{\et,*}\bL)$.
\end{proof}

We first treat an algebraic snc boundary.
We say that an étale $\mathbb Z_p$-local system
has \emph{(quasi-)unipotent geometric monodromy along a divisor},
if its rationalisation has (quasi-)unipotent geometric monodromy
in the sense of~\cite[Definition 6.3.7]{MR4592580}.
The logarithmic Riemann-Hilbert correspondence
sends those étale $\mathbb Z_p$-local system which are de Rham and have unipotent geometric monodromy to integrable log-connections with nilpotent residues~\cite[Theorem~3.2.12(2)]{MR4536903}.

\begin{prop}\label{prop:snc-extension}
Let $j\colon U\hookrightarrow X$ be the complement of an algebraic snc divisor $Z\subset X$, let $\mathbb L$ be an étale $\mathbb Z_p$-local system
on $U$, and write $\E:=D_{\dR}\left(\mathbb L\right)$ for the associated
filtered vector bundle with integrable connection on $U$.  Then $j_+\E=j_*\E$ is a coadmissible, weakly holonomic
$\Dcap_X$-module. Furthermore,
$j_{+}\E$ is holonomic if $\mathbb L$
is de Rham.
\end{prop}

\begin{proof}
The assertion is local on $X$. Corollary~\ref{cor:finite-etale-boundary-charts},
Lemma~\ref{fettest}, and
Lemma~\ref{lem:DdR-comp-directimages}
reduce us to
\[
 X=\Sp K\langle z_1,\dots,z_n\rangle,
 \qquad
 X\setminus U=V(z_1\cdots z_m).
\]
The logarithmic Riemann--Hilbert correspondence
\cite[Theorem~1.7]{MR4536903} extends $\E$ to a logarithmic connection
$\M_{\log}$ on $X$ with rational residue eigenvalues.

Put $\M=\M_{\log}(*D)$. Corollary~\ref{cor:rationalresidues} shows that the
$b$-functions of a suitable generating set of $\M$ have rational roots. The
criterion of Bitoun--Bode~\cite[Theorem~1.2]{Bitoun} therefore implies that
$j_*\E=j_*\mathcal{M}|_U$ is coadmissible and weakly holonomic by Corollary~\ref{postypeext}.

For holonomicity, we note that $\mathbb{L}$
has quasi-unipotent geometric monodromy along $Z$;
this is explained in the proof of~\cite[Lemma 3.4.11]{MR4536903},
see also~\cite[Lemma 2.15]{LiuZhuRH2017}.
Thus we can choose a Kummer cover
\[
 \pi\colon X_N=\Sp K\langle t_1,\dots,t_n\rangle\to X,
 \qquad
 z_i\mapsto
 \begin{cases}
 t_i^N,&i\leq m,\\
 t_i,&i>m
 \end{cases}
\]
with $U_N=\pi^{-1}(U)$ and
$\pi_U=\pi|_{U_N}$, for $N>0$ large enough,
such that $\pi_U^{*}\bL$ has unipotent geometric monodromy along
$Z_N:=X_N\setminus U_N$. Here we view $X_N$ as a log adic space,
with log structure induced by the divisor $Z_N$. 

We observe that the connection $\pi_U^*\mathcal E$ extends to
a log-connection on $X_N$, namely
\[
  \mathcal F:= D_{\dR,\log}\left(\pi_U^*\mathbb L\right),
\]
where we confuse the $\mathbb Z_{p}$-local system $\pi_U^*\mathbb L$
on $U_{N,\et}$ as the associated $\mathbb Z_p$-local
system on the Kummer étale site $X_{N,\ket}$.
Indeed,
\[
  \mathcal F|_{U_{N}}
  =D_{\dR}\left(\pi_{U}^{*}\mathbb L\right)
  \cong\pi_U^* D_{\dR}(\mathbb L)
  =\pi_U^* \mathcal E,
\]
where the isomorphism is~\cite[Theorem 3.9(ii)]{LiuZhuRH2017}.

Since $\pi_U^{*}\bL$ has unipotent geometric monodromy along $Z_N$ and it is de Rham (compare
\cite[Theorem 3.9(ii)]{LiuZhuRH2017})
we can apply~\cite[Theorem~3.2.12(2)]{MR4536903}.
This yields that the logarithmic extension
$\mathcal F$ of $\pi_U^*\E$ has nilpotent residues
along $Z_N$.
Hence $j_{N,+}\pi_U^*\E$,
for $j_N\colon U_N\hookrightarrow X_N$ the embedding,
is holonomic by
Corollary~\ref{cor:nilpotent-log-holonomic}. Since $\pi$ is projective,
\cite[Theorem~5.10(iii)]{Hol} and compatibility of direct images give a
holonomic complex
\[
 \pi_+j_{N,+}\pi_U^*\E
 \cong j_+\pi_{U,+}\pi_U^*\E.
\]
The unit and trace maps exhibit $\E$ as a direct summand of
$\pi_{U,+}\pi_U^*\E$, so that $j_+\E=j_*\E$ is a direct summand of $j_*\pi_{U, +}\pi_U^*\E$. Lemma~\ref{directsummand} now proves that $j_+\E$ is
holonomic.
\end{proof}

\begin{proof}[Proof of Theorem~\ref{introthm:snc}]
For $\mathbb{L}:=R^q f_{\et,*}\mathbb Z_p$,
Lemma~\ref{lem:DdR-comp-directimages} implies
\[
  \mathcal E^q
  \cong D_{\dR}\left(\mathbb L\right).
\]
Since $\mathbb L$ is a de Rham $\mathbb Z_p$-local system
on $U_{\et}$,
Proposition~\ref{prop:snc-extension}
applies.
\end{proof}

\begin{proof}[Proof of Theorem~\ref{thm:main}]
We show the stronger statement that for each $q\in \mathbb{Z}_{\geq 0}$, $j_+\mathcal{E}^q$ is a holonomic $\C$-complex on $X$. As the cohomology groups of $\C$-complexes are coadmissible by \cite[Theorem 8.9(i)]{6Op} and 
\begin{equation*}
    \mathrm{H}^i(j_+\mathcal{E}^q)=R^ij_*\mathcal{E}^q,
\end{equation*}
this is enough to deduce the theorem.

Temkin's desingularisation, in the form used in
\cite[Sections~10.4--10.5]{Dcapthree}, gives locally a projective morphism
$\rho\colon X'\to X$ which is an isomorphism over $U$ and for which
$X'\setminus U$ is snc. If $j'\colon U\hookrightarrow X'$, then
\[
 j_+\E^q\cong\rho_+j'_+\E^q
\]
by \cite[Corollary 9.8]{6Op}.

Proposition~\ref{prop:snc-extension} makes $j'_+\E^q$ holonomic. As projective direct images preserve holonomicity by \cite[Theorem 5.10(iii)]{Hol}, we deduce that $j_+\E^q$ is a holonomic $\C$-complex, as desired.
\end{proof}

\begin{proof}[Proof of Theorem~\ref{introthm:mainhol}]
The complex $f_+\O_Y$ is bounded and has cohomology sheaves $\E^q$. Its
canonical truncation filtration has graded pieces $\E^q[-q]$. After applying
$j_+$, these become the holonomic complexes $j_+\E^q[-q]$ from the proof of
Theorem~\ref{thm:main}. By~\cite[Lemma~5.2]{Hol}, holonomic complexes form a
triangulated subcategory. Hence $j_+f_+\O_Y$ is a holonomic $\C$-complex.
\end{proof}

\bibliographystyle{plain}
\bibliography{lambda}
\end{document}